\documentclass[parskip=half]{scrartcl}
\usepackage[utf8]{inputenc}
\usepackage[T1]{fontenc}
\usepackage{lmodern}
\usepackage{amsmath,amssymb,amsthm,amsfonts,bm,mathtools}
\usepackage{enumerate,color}
\usepackage[usenames,dvipsnames]{xcolor}
\usepackage{hyperref}
\hypersetup{
  colorlinks=true,
  linkcolor=black,
  citecolor=black,
  urlcolor=blue,
  pdftitle={Location of blow-up points in fully parabolic chemotaxis systems with spatially heterogeneous logistic source},
  pdfauthor={Fuest, Lankeit, Mizukami},
  pdfkeywords={},
  bookmarksopen=true,
}
\newcommand{\tops}{\texorpdfstring}

\usepackage[numbers]{natbib}
\makeatletter
\renewenvironment{proof}[1][\proofname]{\par
  \pushQED{\qed}%
  \normalfont \topsep0\p@\relax
  \trivlist
  \item[\hskip\labelsep\scshape
  #1\@addpunct{.}]\ignorespaces
}{%
  \popQED\endtrivlist\@endpefalse
}
\makeatother

\numberwithin{equation}{section} 
\newtheorem{thm}{Theorem}[section]

\newtheorem{lem}[thm]{Lemma}

\theoremstyle{definition}

\newtheorem{remark}[thm]{Remark}
\newtheorem*{prth1.1}{Proof of Theorem 1.1}

\newcommand{\N}{\mathbb N}
\newcommand{\R}{\mathbb R}

\renewcommand{\thefootnote}{\fnsymbol{footnote}} 
\newcommand{\ep}{\varepsilon}

\newcommand{\tmax}{T_{\mathrm{max}}}
\newcommand{\lp}[2]{\|#2\|_{L^{#1}(\Omega)}}

\newcommand{\Lom}[1]{L^{#1}(\Omega)}
\newcommand{\Ombar}{\overline{\Omega}}
\newcommand{\Om}{\Omega}

\newcommand{\ball}[2]{B_{#2}(#1)}
\newcommand{\balla}{\ball{x_0}{5d}}
\newcommand{\ballb}{\ball{x_0}{4d}}
\newcommand{\ballc}{\ball{x_0}{3d}}
\newcommand{\balld}{\ball{x_0}{2d}}
\newcommand{\balle}{\ball{x_0}{d}}

\newcommand{\io}{\int_\Omega}

\newcommand{\ol}{\overline}
\newcommand{\phip}{\varphi_p}
\newcommand{\kapp}{\kappa_0}

\newcommand{\diff}{\mathrm d}

\newcommand{\ds}{\,\diff s}

\newcommand{\hp}{\hphantom}
\newcommand{\pe}{\mathrel{\hp{=}}}

\newcommand{\mc}[1]{\mathcal{#1}}
\newcommand{\ur}[1]{\mathrm{#1}}
\newcommand{\ure}{\ur e}
\newcommand{\ddt}{\frac{\mathrm{d}}{\mathrm{d}t}}

\newcommand{\eps}{\varepsilon}
\newcommand{\intom}{\int_\Omega}

\newcommand{\defs}{\coloneqq}
\newcommand{\sfed}{\eqqcolon}

\usepackage{newunicodechar}
\newunicodechar{ε}{\varepsilon}
\newunicodechar{∇}{\nabla}
\newunicodechar{Δ}{\Delta}
\newunicodechar{μ}{\mu}
\newunicodechar{κ}{\kappa}
\newunicodechar{∂}{\partial}
\newunicodechar{ν}{\nu}
\newunicodechar{ℝ}{\R}
\newunicodechar{π}{\pi}
\newunicodechar{α}{\alpha}
\newunicodechar{δ}{\delta}
\newunicodechar{∞}{\infty}
\newunicodechar{φ}{\varphi}
\newunicodechar{τ}{\tau}
\newunicodechar{γ}{\gamma}
\newunicodechar{Θ}{\theta}

\newcommand{\f}[2]{\frac{#1}{#2}}
\newcommand{\normj}[2][]{\|#2\|_{#1}}
\newcommand{\Ktilde}{\widetilde{K}}
\newcommand{\nn}{\nonumber}

\usepackage{authblk}

\author[1]{Mario~Fuest\footnote{e-mail: fuest@ifam.uni-hannover.de, ORCID: 0000-0002-8471-4451}}
\author[1]{Johannes~Lankeit\footnote{e-mail: lankeit@ifam.uni-hannover.de, ORCID: 0000-0002-2563-7759}}
\author[2]{Masaaki~Mizukami\footnote{e-mail: masaaki.mizukami.math@gmail.com, ORCID: 0000-0002-5496-5129}}

\affil[1]{Leibniz Universität Hannover, Institut für Angewandte Mathematik, \protect\\ Welfengarten 1, 30167 Hannover, Germany}
\affil[2]{Department of Mathematics, Faculty of Education, 
Kyoto University of Education, \protect\\ 1, Fujinomori, Fukakusa, Fushimi-ku, Kyoto 612-8522, Japan} 

\date{}

\title{Location of blow-up points in fully parabolic chemotaxis systems with spatially heterogeneous logistic source}

\begin{document}
\setkomafont{title}{\normalfont\Large}
\maketitle

\renewcommand{\thefootnote}{\fnsymbol{footnote}} 

\KOMAoptions{abstract=true}
\begin{abstract}
\noindent
We consider the fully parabolic, spatially heterogeneous chemotaxis-growth system 
\begin{align*}
  \begin{cases}
    u_t = \Delta u - \nabla\cdot(u\nabla v) + \kappa(x)u-\mu(x)u^2, \\
    v_t = \Delta v - v + u
  \end{cases}
\end{align*}
in bounded domains $\Omega\subset \mathbb{R}^2$ and show that the blow-up set is contained in the set of zeroes of $\mu$.
\end{abstract}


\section{Introduction}
Solutions to semilinear parabolic equations (say, in a domain $\Omega\subset ℝ^n$ with smooth boundary) may cease to exist at finite time $T$. In this case, they usually `blow up' in the sense that
\[
\limsup_{t\nearrow T} \|u(\cdot, t)\|_{L^\infty(\Omega)} = \infty.
\]
Investigations then may be concerned with the question \textit{where} such blow-up happens, i.e.\ with a characterization of the blow-up set
\[
\mathcal{B}=\{\,x\in \Ombar \mid \exists ((x_j,t_j))_{j \in \N} \subset \Ombar \times(0,T): (x_j, t_j)\to (x, T),\; |u(x_j,t_j)|\to \infty \text{ as } j\to\infty\,\},
\]
see e.g.\ \cite{friedman_mcleod,fujita_chen,giga_kohn,galaktionov_vazquez,velazquez_semilinear,fujishima_ishige} for semilinear parabolic equations or \cite{souplet_singlepoint} for systems.

The Keller--Segel equations of chemotaxis 
\begin{equation}\label{KS}
 u_t=Δu-∇\cdot(u∇v),\qquad v_t=Δv-v+u
\end{equation}
constitute a well-studied system (see e.g.\ the surveys \cite{BBTW,lanwin_survey}), whose ability to admit blow-up is not only of mathematical but also of biological interest and has endowed it with a firm position as prototype of spontaneous aggregation and thus structure formation. Its basic mechanisms -- production of a signal substance (of concentration $v$) and partially undirected (heat equation) and partially directed motion (taxis term, $-∇\cdot (u∇v)$) of cells (with density $u$) -- are fundamental throughout biology, and the system thus emerges -- possibly as submodel of more complex descriptions -- in a variety of contexts, \cite{hillen_painter_survey,painter_selforganisation}.

If \eqref{KS} is considered under no-flux boundary conditions in a bounded, spatially two-dimensional domain with smooth boundary and with sufficiently regular initial data $(u_0,v_0)$, it has local classical solutions -- which necessarily are global if $\io u_0<4π$, \cite{NSY}. If, however, the initial mass is larger, the solution may blow up; and this happens for a large set of initial data, \cite{herrero_velazquez,horstmann_wang,mizowin}. 
The first such solution found was the radially symmetric solution constructed in \cite{herrero_velazquez}, which aggregates towards $8πδ_0$ (with $δ_0$ denoting the Dirac measure centred at the origin); and, in fact, for radially symmetric solutions, the origin is the only possible blow-up point, i.e.\ $\mc B \subseteq \{0\}$, \cite{nagai_senba_suzuki}. If the Lyapunov functional $\io u\ln u - \io uv + \f12 \io |∇v|^2 + \f12\io v^2$ remains bounded from below, the points of $\mc B$ are isolated also in non-radial scenarios. 
At every isolated blow-up point $x_0$, the solution collapses into $mδ_{x_0}$ (plus an $L^1$-function) in the sense of weak-star convergence in the space of Radon measures, \cite{nagai_senba_suzuki},
and where $m=4π$ or $m=8π$, depending on whether $x_0\in ∂\Om$ or $x_0\in \Om$. 

While homogeneous Dirichlet boundary conditions for $v$ in a parabolic--elliptic simplification of \eqref{KS} can rule out boundary blow-up, see \cite{suzuki_exclusion}, without these modifications, blow-up in a boundary point is still possible for \eqref{KS}, as demonstrated in \cite{FuestLankeitCornersCollapseSimple2023} (for less smooth domains, like a quarter disk).
Further estimates of solutions near blow-up, including a nondegeneracy result for blow-up points analogous to \cite{giga_kohn}, can be found in \cite{mizoguchi_souplet}.  

If additional source terms of logistic type $+κu-μu^2$ are included in the first equation of \eqref{KS} -- for example, in order to account for population growth, \cite[Sec.~2.8]{hillen_painter_survey} --, the possibility of blow-up may be significantly impaired. 
This question is not completely settled in higher-dimensional cases, see \cite{lan_evsmooth,viglialoroVeryWeakGlobal2016,WinklerSolutionsParabolicKeller2023,YanFuestWhenKellerSegel2021,ding_lankeit} for the construction of weaker solutions and their eventual smoothness properties or \cite{WinklerBoundednessHigherdimensionalParabolic2010,XiangHowStrongLogistic2018,XiangChemotacticAggregationLogistic2018} for global existence theorems requiring largeness of $μ$, but \cite{WinklerBlowupHigherdimensionalChemotaxis2011, WinklerFinitetimeBlowupLowdimensional2018,FuestApproachingOptimalityBlowup2021} for some indications of possible unboundedness at least in parabolic--elliptic system variants or  \cite{WinklerEmergenceLargePopulation2017,wang_win_xiang_CVPDE} 
for large population densities even in the absence of blow-up. In $\Om\subset ℝ^2$, on the other hand, mere positivity of the constant $μ$ suffices to guarantee globally bounded solutions, \cite{OsakiEtAlExponentialAttractorChemotaxisgrowth2002}. 

The same does not automatically carry over to spatially inhomogeneous (i.e.\ non-constant) $κ$ and $μ$. Such inhomogeneity can be needed for more realistic inclusion of environmental conditions, for example in spatial ecology (and can have significant effects, as the results on (non-)coexistence in the competition model in  \cite{DockeryEtAlEvolutionSlowDispersal1998} show), and has been studied in a related chemotaxis model in, e.g., \cite{salako_shen_III}. (See also references therein.) The following thus is the system we will consider in the present article.

\begin{align}\label{P}
 \begin{cases}
u_t=\Delta u -\nabla \cdot (u\nabla v) + \kappa(x)u -\mu (x)u^2,  
         &x\in \Omega,\ t>0, 
 \\ 
  v_t = \Delta v - v + u, 
         &x\in \Omega,\ t>0,
 \\
 \partial_\nu u = \partial_\nu v = 0,
      & x\in \partial \Omega, \ t>0, 
 \\
 u (x,0) = u_0 (x), \ 
 v(x,0) = v_0 (x), 
 & x\in \Omega,
 \end{cases}
\end{align}
where $\kappa,\mu\in C^0(\ol{\Omega})$ with $\mu\ge 0$ in $\ol{\Omega}$ and where $u_0 \in C^0 (\ol{\Omega})$ and $v_0 \in W^{1,\infty}(\Omega)$ are nonnegative initial data. 

For this system, existence of solutions was shown in \cite{YanFuestWhenKellerSegel2021}, within a generalized concept of solvability (of $u$ being a `weak logarithmic supersolution' and a `mass subsolution' and $v$ being a weak solution; cf.\ \cite[Def.~2.1]{YanFuestWhenKellerSegel2021}) under some restrictions on the growth of $μ$, which for the quadratic exponent in \eqref{P} are fulfilled whenever $μ(x)\le |x|^{α}$ for some $α<2$.

On the other hand, if the equation for $v$ in \eqref{P} is replaced by an elliptic counterpart, blow-up can occur  \cite{FuestFinitetimeBlowupTwodimensional2020} (see also \cite{BlackEtAlRelaxedParameterConditions2021} for possibly nonlinear diffusivities), and again the question arises, \textit{where} this blow-up may take place. For the parabolic--elliptic variant of \eqref{P}, this question was answered in \cite{BlackEtAlPossiblePointsBlowup2023} with the conclusion that $\mc B \subseteq \mu^{-1}(\{0\})$, i.e.\ that only zeroes of $μ$ are possible points of blow-up. It is the aim of the present paper to prove a corresponding result for the fully parabolic analogue \eqref{P}:

\begin{thm}\label{mainthm1}
Let $\Omega\subset \mathbb{R}^2$ be a smooth, bounded domain, $\kappa,\mu \in C^0(\ol{\Omega})$ 
with $\mu\ge 0$, $0 \le u_0 \in C^0(\Ombar)$, $0 \le v_0 \in W^{1, \infty}(\Omega)$, $\tmax \in (0, \infty]$
and let $(u, v) \in \left( C^{0}(\Ombar \times [0, \tmax)) \cap C^{2, 1}(\Ombar \times (0, \tmax)) \right)^2$ be a nonnegative, classical solution of \eqref{P}.
\begin{enumerate}
  \item[(i)]
    Let $x_0 \in \ol{\Omega}$ with $\mu(x_0) >0$ and $T\in (0,\tmax]\cap (0,\infty)$. Then there exist an open neighbourhood $U$ of $x_0$ in $\Ombar$ and $C > 0$ such that
    \[
      \|u(\cdot,t)\|_{L^\infty(U)} \le C \quad \text{for all } t\in (0,T). 
    \]

  \item[(ii)]
    If $\tmax < \infty$, then the blow-up set 
    \[
    \mathcal{B} \defs \{\, x\in \ol{\Omega} \mid \exists ((x_j, t_j))_{j\in \mathbb{N}}\subset \ol{\Omega} \times (0,\tmax): (x_j, t_j) \to (x, \tmax),\; u(x_j,t_j)\to \infty \text{ as } j\to\infty\,\}
    \]
    is contained in $\{\,x\in \ol{\Omega} \mid \mu(x) =0\,\}$. 
\end{enumerate}
\end{thm}

\begin{remark}
 At least as soon as $κ$ and $μ$ are Hölder continuous, local existence of solutions of the regularity assumed in Theorem~\ref{mainthm1} is achievable by standard techniques, cf.\ the fixed point construction in \cite[Lemma~3.1]{BBTW} together with classical parabolic theory \cite[Thm.~IV.5.3]{LSU}; if they are not global in time, 
 \[
  \limsup_{t\nearrow \tmax} \normj[\Lom{∞}]{u(\cdot,t)}=∞, 
 \]
 and hence $\mathcal{B}\neq\emptyset$ in Theorem~\ref{mainthm1} (ii).
\end{remark}

In systems such as \eqref{P}, uniform-in-time bounds for $u$ in $L^p(\Omega)$  can be bootstrapped to $L^\infty$ bounds whenever $p > \frac{\dim\Omega}{2}$ (cf.\ \cite[Lemma~3.2]{BBTW}).
Indeed, the global existence proof in \cite{OsakiEtAlExponentialAttractorChemotaxisgrowth2002} for functions $\mu$ with a positive lower bound
first shows boundedness of $\intom u \ln u$ and then of $\intom u^2$, yielding the desired $L^\infty$ bound.

When $\mu$ is allowed to have zeros, however, the techniques from \cite{OsakiEtAlExponentialAttractorChemotaxisgrowth2002} are not directly applicable.
In Lemma~\ref{lem;DifIneqforU} and Lemma~\ref{lem;LpForU}, we instead consider the evolution of functionals $\intom \varphi u^p$ for some $\varphi$ localized near a point $x_0$ and small $p > 1$ --
similarly as in the parabolic--elliptic counterpart \cite{BlackEtAlPossiblePointsBlowup2023} of the present paper.
With bounds for this quantity at hand, we can then ensure boundedness of $u$ near $x_0$ (see Lemma~\ref{lem;LqForV} and Lemma~\ref{lem;LinftyForU}).

In contrast to \cite{BlackEtAlPossiblePointsBlowup2023}, the second equation in \eqref{P} is parabolic, meaning that replacing $-\Delta v$ by $u - v$ in the calculations is no longer possible.
Instead, we apply maximal Sobolev regularity results (see Lemma~\ref{lem;MaxSobReg} and Lemma~\ref{lem;MaxSobRegTestver}) to $\varphi v$.
When applied with $\varphi \equiv 1$, maximal Sobolev regularity is known to be useful in fully parabolic chemotaxis equations.
It has, for instance, been employed in \cite{CaoLargeTimeBehavior2017} in giving a short proof for the large-time behaviour of solutions to a system with a spatially homogeneous logistic source.
When $\varphi \not\equiv 1$, as is necessitated by our localization procedure, however, the right-hand side of the corresponding inequality contains a term involving $\nabla v$.
We control it by again relying on maximal Sobolev regularity in Lemma~\ref{lm:Delta_v_local} and Lemma~\ref{lem;L4-epEstiForV},
which in turn are based on the spatially local $L^2$-$L^2$ bound for $u$ and spatially global estimates for $v$ and $\nabla v$ obtained in Lemma~\ref{lem;L1ForU} and Lemma~\ref{lem;EstiForV}, respectively.

In order to absorb the resulting expressions by dissipative terms of the functional $\intom \varphi u^p$,
we need to ensure that $p-1$ is smaller than a certain constant (see the discussion directly before Lemma~\ref{lem;DifIneqforU}).
Obviously, the latter may not depend in an unfavourable way on $p$ and hence it is essential to trace the dependency of constants throughout our argument.
In particular, our approach requires that the optimal constant in the maximal Sobolev regularity estimate is bounded on compact subsets of $(1, \infty)^2$,
a fact we prove in Lemma~\ref{lem;MaxSobReg}.

\paragraph{Plan of the paper.}
Quite general maximal regularity results and the existence of suitable cutoff functions are collected in Section~\ref{sec2},
while Section~\ref{sec3} and Section~\ref{sec4} deal with spatially global and spatially local bounds of solutions to \eqref{P}, respectively.
The order of the lemmata in the latter is nearly reversed from the line of reasoning outlined above.
This is due to the fact that we often consider cutoff functions whose support is contained in the set where the cutoff function from the previous lemma equals $1$
and that we choose to begin with arguments where the corresponding cutoff function has the largest support.

\section{Preliminaries}\label{sec2}
As already mentioned, maximal regularity properties of the scalar parabolic equation
\begin{align}\label{prob:w}
  \begin{cases}
    w_t=\Delta w -w + f & \text{in $\Omega\times(0,T)$}, \\ 
    \partial_\nu w  = 0 & \text{on $\partial \Omega\times(0,T)$}, \\
    w(\cdot,0) = w_0    & \text{in $\Omega$},
 \end{cases}
\end{align}
form an essential ingredient in our proof.
To formulate the estimates, we first introduce the Banach spaces
\begin{align*}
  W_N^{2, q}(\Omega) &\defs \{\, \varphi \in W^{2, q}(\Omega) \mid \partial_\nu \varphi = 0 \text{ on $\partial \Omega$ in the sense of traces}\,\}
\intertext{and}
  \mc W^{p, q}(T, \Omega) &\defs W^{1,p}((0,T);L^q(\Omega))\cap L^p((0,T);W_N^{2,q}(\Omega)),
\end{align*}
where $p, q \in (1, \infty)$, $n \in \N$, $\Omega \subset \R^n$ is a smooth, bounded domain and $T \in (0, \infty]$.
We recall a result on unique solvability of \eqref{prob:w} for appropriate data.
\begin{lem}\label{lm:ex_w}
  Let $n \in \N$, $\Omega \subset \R^n$ be a smooth, bounded domain, $T \in (0, \infty]$ and $p, q \in (1, \infty)$.
  For all $w_0\in W_N^{2,q}(\Omega)$ and all $f \in L^p((0, T), L^q(\Omega))$,
  there exists a unique solution $w \in \mc W^{p, q}(T, \Omega)$ of \eqref{prob:w}.
\end{lem}
\begin{proof}
  This follows from \cite[Theorem~3.1]{HieberPrussHeatKernelsMaximal1997}, for instance.
\end{proof}
 
We next state a maximal regularity result for \eqref{prob:w} in $L^p$-$L^q$ spaces, where the optimal constant $K(p, q)$ depends on $(p, q)$.
For its key application in the present paper
we will need that $K$ is bounded on compact subsets of $(1, \infty)^2$ so that for instance $(p-1)K(p+1,p+1) \to 0$ as $p \to 1$, cf.\ the application of Lemma~\ref{lem;MaxSobRegTestver} in \eqref{choices} that the proof of Lemma~\ref{lem;LpForU} relies on.
A short argument based on complex interpolation shows that $K$ is actually continuous.
\begin{lem}\label{lem;MaxSobReg}
  Let $n \in \N$, $\Omega \subset \R^n$ be a smooth, bounded domain and $T \in (0, \infty]$.
  For $p, q \in (1, \infty)$, denote by $\mc K(p, q)$ the set of all numbers $\widehat K(p, q) > 0$ with the following property:
  For all $f\in L^p((0, T); L^q(\Omega))$, 
  the unique solution $w\in \mc W^{p, q}(T, \Omega)$ of \eqref{prob:w} with $w_0 \equiv 0$ given by Lemma~\ref{lm:ex_w} fulfills
  \begin{align}\label{eq:MaxSobReg:ineq}
    \int_0^t \left( \int_\Omega |w|^q \right)^\frac pq + \int_0^t \left( \int_\Omega |w_t|^q \right)^\frac pq + \int_0^t \left( \int_\Omega |\Delta w|^q \right)^\frac pq
  \le
    \big(\widehat K(p,q)\big)^p \int_0^t \left( \int_\Omega |f|^q \right)^\frac pq
  \end{align}
  for all $t \in (0, T]$.
  Then $\mc K(p, q)$ is nonempty and the mapping $(1, \infty)^2 \ni (p, q) \mapsto K(p, q) \defs \inf \mc K(p, q)$ is continuous.
\end{lem}
\begin{proof}  
  As the left-hand side in \eqref{eq:MaxSobReg:ineq} is increasing in $t$ and $f$ can simply be extended by $0$ if necessary,
  it is sufficient to consider the case $t=T= \infty$ in seeing that $\mc K(p, q) \neq \emptyset$.
  This is shown by \cite[Theorem~3.1]{HieberPrussHeatKernelsMaximal1997} which asserts that \eqref{eq:MaxSobReg:ineq} holds for all $p, q \in (1, \infty)$ and some $\widehat K(p, q) > 0$.

  If, for arbitrary $T \in (0, \infty]$, we abbreviate $X^{p,q} \defs L^p((0, T); L^q(\Omega))$, let $\iota \colon \mc W^{p, q}(T, \Omega) \to (X^{p,q})^3$ denote the isometry $\varphi \mapsto (\varphi, \varphi_t, \Delta \varphi)$ and let $A$ be the (linear) operator mapping $f$ to the unique solution $w$ of \eqref{prob:w} given by Lemma~\ref{lm:ex_w}, then \eqref{eq:MaxSobReg:ineq} thus for all $t \in (0, T)$ holding with the choice $\widehat K(p, q) = K(p, q)$ shows that $\normj{\iota\circ A}=K(p,q)$ is the norm of $\iota\circ A\colon X^{p,q}\to (X^{p,q})^3$ for any $p,q\in(1,\infty)$. 
  Since $X^{p_{Θ},q_{Θ}}$ coincides with the interpolation space $[X^{p_0,q_0},X^{p_1,q_1}]_{Θ}$  for $p_0, p_1, p_\theta, q_0, q_1, q_\theta \in (1, \infty)$ and $\theta \in (0, 1)$
  with $\frac{1}{p_\theta} = \frac{1-\theta}{p_0} + \frac{\theta}{p_1}$ and $\frac{1}{q_\theta} = \frac{1-\theta}{q_0} + \frac{\theta}{q_1}$ (cf.\ \cite[Theorem~1.18.4]{TriebelInterpolationTheoryFunction1978}), complex interpolation applied to $\iota \circ A$ yields
  $K(p_\theta, q_\theta) \le (K(p_0, q_0))^{1-\theta} (K(p_1, q_1))^{\theta}$ (cf.\ \cite[Theorem~1.9.3 (a), Definition~1.2.2.2]{TriebelInterpolationTheoryFunction1978}).  
  That is, $\Phi \colon (0, 1)^2 \to \mathbb R$, $(\tilde p, \tilde q) \mapsto \ln K(\frac1{\tilde p}, \frac{1}{\tilde q})$ is convex.
  Since $(0, 1)^2$ is open, $\Phi$ is continuous and so is $K = {\exp} \circ \Phi \circ ((p, q) \mapsto (\frac1p, \frac1q))$.
\end{proof}

As a straightforward consequence of Lemma~\ref{lem;MaxSobReg}, we obtain an $L^p$-$L^p$ bound for $\varphi \Delta w$ in terms of the data and the cutoff function $\varphi$.
In particular, the following lemma is applicable in situations where $L^p$-$L^p$ bounds are available for $\varphi f$ but not for $f$.
\begin{lem}\label{lem;MaxSobRegTestver}
  Let $n \in \N$, $\Omega \subset \R^n$ be a smooth, bounded domain and $T \in (0, \infty]$.
  There exists a continuous function $\widetilde K : (1, \infty) \to (0, \infty)$
  such that for all $w_0 \in W_N^{2,p}(\Omega)$ and all $f\in L^p(\Omega \times (0, T))$,
  the solution $w\in \mc W^{p, p}(T, \Omega)$ of \eqref{prob:w} given by Lemma~\ref{lm:ex_w} satisfies  
  \begin{align*}
   \int_0^t \int_\Omega |\varphi \Delta w|^p
  &\le 
    \widetilde{K}(p) \left\{ \|\varphi w_0\|_{W^{2,p}(\Omega)}^p
    + \int_0^t \int_\Omega \Big(|\nabla \varphi\cdot \nabla w|^p + |w \Delta \varphi|^p + | \varphi f|^p\Big)  \right\}
  \end{align*}
  for all $\varphi \in C^2(\Ombar)$ with $\partial_\nu \varphi = 0$ on $\partial \Omega$, all $t \in (0, T)$ and all $p \in (1, \infty)$.
\end{lem}
\begin{proof}
  Since for any $\varphi \in C^2(\Ombar)$ with $\partial_\nu \varphi = 0$ on $\partial \Omega$, 
  \begin{align*}
    \Delta (\varphi w - \ure^{-t}\varphi w_0) = \varphi \Delta w + 2 \nabla \varphi \cdot \nabla w + w\Delta \varphi - \ure^{-t} \Delta (\varphi w_0) \qquad \text{in $\Omega \times (0, T)$}
  \end{align*}
  and $-2 \nabla \varphi\cdot \nabla w - w \Delta \varphi + \ure^{-t}\Delta(\varphi w_0)  + \varphi f \in L^p(\Omega \times (0, T))$,   the function $z \defs (w - \ure^{-t}w_0)φ$ is the unique solution of
  \begin{align*}
    \begin{cases}
      z_t=\Delta z -z -2 \nabla \varphi\cdot \nabla w - w \Delta \varphi + \ure^{-t}\Delta(\varphi w_0) + \varphi f & \text{in $\Omega\times(0,T)$}, \\
      \partial_\nu z  = 0                                                                                                & \text{on $\partial \Omega\times(0,T)$}, \\
      z(\cdot,0) = 0                                                                                                     & \text{in $\Omega$}
   \end{cases}
  \end{align*}
  in the sense of Lemma~\ref{lm:ex_w}.
  Moreover, we see from Lemma~\ref{lem;MaxSobReg} that (with $K$ as in that lemma)
  \begin{align*}
   \int_0^t \int_\Omega |\varphi \Delta w|^p
   & = \int_0^t \int_\Omega \Big|\Delta (\varphi (w-\ure^{-s}w_0))  -2\nabla \varphi\cdot \nabla w - w \Delta \varphi +\ure^{-s}Δ(φw_0)\Big|^p\ds
   \\
   & \le 
   2^{p-1} \int_0^t \int_\Omega |\Delta z|^p + 
   2^{p-1} \int_0^t \int_\Omega \Big| -2\nabla \varphi\cdot \nabla w -w \Delta \varphi +\ure^{-s}Δ(φw_0)\Big|^p\ds
   \\
  &\le 2^{p-1} K^p(p, p)
  \int_0^t \int_\Omega \Big|-2 \nabla \varphi\cdot \nabla w - w \Delta \varphi + \ure^{-s}\Delta(\varphi w_0) + \varphi f\Big|^p \ds
  \\
  &\pe +\, 6^{p-1} \int_0^t \int_\Omega \left(|2\nabla \varphi\cdot \nabla w|^p + |w \Delta \varphi|^p +\ure^{-ps}|Δ(φw_0)|^p\right)\ds
  \\
  &\le 
    \widetilde{K}(p) \left\{ \|\varphi w_0\|_{W^{2,p}(\Omega)}^p + \int_0^t \int_\Omega \Big( |\nabla \varphi\cdot \nabla w|^p  +|w \Delta \varphi|^p +  |\varphi f|^p \Big)  \right\}
  \end{align*}
  holds for all $\varphi \in C^2(\Ombar)$ with $\partial_\nu \varphi = 0$ on $\partial \Omega$, all $t \in (0, T)$ and all $p \in (1, \infty)$,
  where $\widetilde{K}(p) \defs (8^{p-1}K^p(p,p)+6^{p-1})2^p$. 
  Continuity of $\widetilde K$ follows from the continuity of $K$ asserted by Lemma~\ref{lem;MaxSobReg}.
\end{proof}

Inter alia in order to warrant applicability of Lemma~\ref{lem;MaxSobRegTestver}, we need to ensure that there are cutoff functions whose normal derivatives vanish at the boundary.
Having in mind that the point $x_0$ in Theorem~\ref{mainthm1}~(i) may lie on the boundary,
we cannot confine ourselves to cutoff functions supported in the interior of $\Omega$,
which would not only trivially fulfill homogeneous Neumann boundary conditions but would also be easily obtained by a mollification argument.
Instead, we recall that cutoff functions suitable for our purposes have been constructed in \cite{BlackEtAlPossiblePointsBlowup2023}.
\begin{lem}\label{lem;CutOffF}
Let $\Omega \subset \mathbb{R}^2$ be a smooth, bounded domain and $x_0\in \Ombar$.
Then there is a relatively open set $W\subset \Ombar$ containing $x_0$ such that 
whenever $A\subset W$ is compact, $V$ is a open neighbourhood of $A$ in $W$ and $\eta\in (0,\frac 12)$,
then there exist $\varphi\in C^2(\Ombar)$ with $\varphi^\eta \in C^2(\Ombar)$ and $C_\varphi>0$ such that 
\begin{align}\label{Condi;testfunction}
  0\le \varphi \le 1 \text{ in } \Ombar,
\quad 
  \varphi = 0 \text{ in } \Ombar \setminus V,
\quad 
  \varphi =1   \text{ in } A,
\quad 
  \partial_\nu \varphi = 0 \text{ on } \partial \Omega
\end{align}
and 
\begin{align}\label{eq:cutoff:D_varphi_est}
  |\nabla \varphi| \le C_\varphi \varphi^{1-\eta} \quad \text{and} \quad  |\Delta \varphi|\le C_\varphi \varphi^{1-2\eta} \quad  \text{in } \Ombar
\end{align}
\end{lem}
\begin{proof}
  This is essentially \cite[Lemma~3.3]{BlackEtAlPossiblePointsBlowup2023}; the additional claim $\varphi^\eta \in C^2(\ol G)$ follows directly from the proof there.
\end{proof}

\section{Spatially global estimates}\label{sec3}
Throughout this section, let
\begin{align}\label{eq:ass_basic}
  \begin{cases}
    \Omega \subset \mathbb R^2 \text{ be a smooth, bounded domain}, \\ 
    \kappa, \mu \in C^0(\Ombar) \text{ with $\mu \ge 0$ in $\Ombar$}, \\
    0 \le u_0 \in C^0(\Ombar),\, 0 \le v_0 \in W^{1, \infty}(\Omega), \, \tmax \in (0, \infty], \\
    (u, v) \in \left( C^{0}(\Ombar \times [0, \tmax)) \cap C^{2, 1}(\Ombar \times (0, \tmax)) \right)^2 \\
    \hphantom{(u, v)}\; \text{be a nonnegative, classical solution of \eqref{P}}.
  \end{cases}
\end{align}

First a~priori estimates for the first solution component are rapidly obtained by integrating the corresponding equation in \eqref{P}.
Unlike in situations where $\kappa$ and $\mu$ are positive constants, however,
the following bounds -- and hence also all subsequent ones -- are necessarily local in time;
for instance, the choices $\kappa \equiv 1$ and $\mu \equiv 0$ are allowed by \eqref{eq:ass_basic} and imply $\intom u(\cdot, t) = \ure^t \intom u_0$ for $t \in (0, \tmax)$.
\begin{lem}\label{lem;L1ForU}
  Assume \eqref{eq:ass_basic}.
  For all $T\in (0,\tmax]\cap (0,\infty)$, there is $C>0$ such that 
  \begin{align}\label{L1ForU}
    \lp{1}{u(\cdot,t)}\le C
    \qquad \text{for all $t\in (0,T)$}
  \end{align}
  and
  \begin{align}\label{L2L2ForU}
    \int_0^T \int_\Omega \mu u^2 \le C.
  \end{align}
\end{lem}
\begin{proof}
  Integrating the first equation in \eqref{P} shows
  \begin{align*}
    \ddt  \io u = \io \kappa u - \io \mu u^2 
    \le \lp{\infty}{\kappa} \io u - \io \mu u^2 \qquad \text{ in } (0,T),
  \end{align*}
 hence $z(t)\defs \lp{1}{u(\cdot,t)} + \int_0^t \ure^{\|\kappa\|_{L^\infty(\Omega)} (t-s)} \io \mu u^2(\cdot, s) \ds$ solves $z'\le \normj[\Lom\infty]{κ} z$ in $(0,T)$, so that 
  \begin{align}\label{ineq;forL1}
    \pe  \lp{1}{u(\cdot,t)} + \int_0^t \io \mu u^2 \le z(t)
    \le  \lp{1}{u_0} \ure^{\lp{\infty}{\kappa}t}
     \le  \lp{1}{u_0} \ure^{\lp{\infty}{\kappa}T}\sfed C
  \end{align}
  for all $t\in (0,T)$.
  This directly yields \eqref{L1ForU},
  while \eqref{L2L2ForU} follows from \eqref{ineq;forL1} and the monotone convergence theorem upon taking $t\nearrow T$.
\end{proof}
%
On the basis of \eqref{L1ForU}, well-known estimates for the inhomogeneous heat equation yield the following a~priori estimates for $v$.
\begin{lem}\label{lem;EstiForV}
  Assume \eqref{eq:ass_basic} and let $T\in (0,\tmax]\cap (0,\infty)$.
  For all $s\in [1, \infty)$ and all $q \in [1, 2)$, there is $C > 0$ such that
  \begin{align}\label{LpForV}
    \lp{s}{v(\cdot,t)} \le C
    \qquad \text{for all $t \in (0, T)$}
  \end{align}
  and
  \begin{align}\label{LqForNablaV}
    \|\nabla v(\cdot,t)\|_{L^q(\Omega)} \le C
    \qquad \text{for all $t \in (0, T)$}.
  \end{align}
\end{lem}
\begin{proof}
  Due to \eqref{L1ForU}, this is a routine application of standard semigroup estimates (as collected in \cite[Lemma~1.3]{WinklerAggregationVsGlobal2010}, for instance).
\end{proof}

\section{Spatially localized estimates}\label{sec4}
Next, we shall gain stronger a~priori estimates near points in $\Ombar$ where $\mu$ is positive, where we can utilise \eqref{L2L2ForU}.
To that end, we henceforth assume that
\begin{align}\label{eq:ass_x0}
  \begin{cases}
    \eqref{eq:ass_basic} \text{ holds}, \\
    T \in (0, \tmax] \cap (0, \infty),\, \tau \in (0, T), \\
    x_0 \in \Ombar \text{ is such that } \mu(x_0) > 0, \\
    d > 0 \text{ and } \mu > \frac{\mu(x_0)}{2} \text{ in } \balla \cap \Ombar, \\
    \balla \cap \Ombar \text{ is contained in the set $W$ given by Lemma~\ref{lem;CutOffF}},
  \end{cases}
\end{align}
where the last two properties can be achieved by choosing $d$ sufficiently small.
 
\subsection{Localized space-time gradient bounds for \tops{$v$}{v}}
Our first spatially local bound is a consequence of maximal Sobolev regularity (in particular Lemma~\ref{lem;MaxSobRegTestver})
and the $L^2$ space-time bound \eqref{L2L2ForU} for the first solution component.
While in the parabolic--elliptic setting considered in \cite{BlackEtAlPossiblePointsBlowup2023}
one immediately obtains $\int_0^T \int_{\{\mu > \mu(x_0)/2\}} |\Delta v-v|^2 = \int_0^T \int_{\{\mu > \mu(x_0)/2\}} u^2 \le C$,
we apparently have to confine ourselves to $L^{2-\eta}$ space-time bounds for $\Delta v$ with $\eta>0$, 
at least at this point in the proof. Moreover, these bounds may still depend on all data fixed in \eqref{eq:ass_x0} and \eqref{eq:ass_basic}, in particular on $x_0$, $τ$ and $T$.
\begin{lem}\label{lm:Delta_v_local}
  Assume \eqref{eq:ass_x0} and let $r \in (1, 2)$.
  Then
  \begin{align}\label{eq:Delta_v_local:statement}
    \int_{\tau}^T \int_{\ballb \cap \Omega} |\Delta v|^r < \infty.
  \end{align}
\end{lem}
\begin{proof}
  We let $A \defs \overline{\ballb} \cap \Ombar$ and $V \defs \balla \cap \Ombar$,
  so that Lemma~\ref{lem;CutOffF} guarantees the existence of a cut-off function $\varphi \in C^2(\Ombar)$ obeying  \eqref{Condi;testfunction}.
  According to Lemma~\ref{lem;MaxSobRegTestver}, we can thus find $c_1=\Ktilde(r)>0$, and from the regularity of $φ$ and $v(\cdot,τ)$ according to \eqref{eq:ass_basic} consequently have $c_2 > 0$, satisfying
  \begin{align*}
    \int_{\tau}^T \int_{A} |\Delta v|^r 
    & =  \int_{\tau}^T \int_{A} |\varphi \Delta v|^r 
     \le \int_{\tau}^T \int_{\Omega} |\varphi \Delta v|^r 
  \\
  &\le c_1 
    \left\{ 
      \|\varphi v(\cdot,{\tau})\|_{W^{2,r}(\Omega)}^r 
      + \int_{\tau}^T \int_\Omega \left( |\nabla \varphi \cdot \nabla v|^r + |v\Delta \varphi|^r + |\varphi u|^r\right)
    \right\}
  \\
  &\le c_2 \left\{ 
    1+ \int_{\tau}^T \int_\Omega |\nabla v|^r + \int_{\tau}^T \int_\Omega |v|^r + \int_{\tau}^T \int_\Omega |\varphi u|^r
  \right\}.
  \end{align*}
  The first two integrals on the right-hand side are finite by \eqref{LqForNablaV} and \eqref{LpForV}, respectively.
  For the remaining one, we make use of the defining property of $d$ in \eqref{eq:ass_x0} and Lemma~\ref{lem;L1ForU}, which show
  \begin{align*}
        \int_{\tau}^T \int_\Omega |\varphi u|^2
    \le \frac{2}{\mu(x_0)} \cdot \frac{\mu(x_0)}{2} \int_\tau^T \int_{V} u^2
    \le \frac{2}{\mu(x_0)} \int_\tau^T \io \mu u^2
    <   \infty.
  \end{align*}
  Thus we obtain \eqref{eq:Delta_v_local:statement}.
\end{proof}

By interpolating between \eqref{LqForNablaV} and \eqref{eq:Delta_v_local:statement} and applying another cutoff argument,
we obtain spatially local space-time bounds for $\nabla v$.
\begin{lem}\label{lem;L4-epEstiForV}
  Assume \eqref{eq:ass_x0}.
  For all $\alpha\in [1,4)$,
  \begin{align}\label{eq:L4-epEstiForV:statement}
    \int_{\tau}^T \int_{\ballc\cap \Omega} |\nabla v|^\alpha < \infty.
  \end{align}
\end{lem}
\begin{proof}
Without loss of generality, we assume $α>\f32$.
  According to Lemma~\ref{lem;CutOffF}, there exists $\varphi \in C^2(\Ombar)$ fulfilling \eqref{Condi;testfunction}
  with $A \defs \overline{\ballc} \cap \Ombar$ and $V \defs \ballb \cap \Ombar$.
  The definitions $r \defs \sqrt{1+2\alpha} - 1 \in (1, 2)$ and $\theta \defs \frac{\frac1r - \frac1\alpha}{\frac12 - \frac1r + \frac1r} = \frac{2\alpha - 2r}{\alpha r} \in (0, 1)$
  imply $\theta \alpha = r$.
  By the Gagliardo--Nirenberg inequality and
  elliptic regularity (cf.\ \cite[Theorem~19.1]{FriedmanPartialDifferentialEquations1976}; note that $\partial_\nu (\varphi v) = 0$ on $\partial V \times (\tau, T)$ by \eqref{P} and \eqref{Condi;testfunction}),
  there are $c_1, c_2 > 0$ such that
  \begin{align*}
          \|\nabla (\varphi v) \|_{L^\alpha(V)}^\alpha
    &\le  c_1 \|D^2 (\varphi v) \|_{L^r(V)}^{\theta \alpha} \|\nabla (\varphi v) \|_{L^r(V)}^{(1-\theta)\alpha}
          + c_1 \|\nabla (\varphi v) \|_{L^r(V)}^{\alpha} \\
    &\le  c_2 \left(\|\Delta (\varphi v) \|_{L^r(V)}^{r} +\normj[L^r(V)]{φv}^r \right)\|\nabla (\varphi v) \|_{L^r(V)}^{(1-\theta)\alpha}
          + c_1 \|\nabla (\varphi v) \|_{L^r(V)}^{\alpha}
  \end{align*}
  in $(\tau, T)$.
  Since $\varphi \in C^2(\Ombar)$
  and the expressions $\int_{\tau}^T \int_V |\Delta v|^r$, $\sup_{t \in (\tau, T)} \intom |\nabla v|^r$ and $\sup_{t \in (\tau, T)} \intom |v|^r$
  are finite by \eqref{eq:Delta_v_local:statement}, \eqref{LqForNablaV} and \eqref{LpForV}, respectively,
  we conclude the existence of $c_3 > 0$ with $\int_{\tau}^T \intom |\nabla (\varphi v)|^\alpha \le c_3$.
  As $\varphi = 1$ in $A$, we obtain \eqref{eq:L4-epEstiForV:statement}.
\end{proof}

\subsection{The functional \tops{$\intom \varphi_p u^p$}{int phi u\^{}p}}
Our next goal is to bound $\intom \varphi_p u^p$ for sufficiently small $p > 1$ and a suitable cutoff function $\varphi_p$.
To ensure that this is indeed possible, we fix $p$ and $\varphi_p$ already at this stage and consider the functional only for these choices.

To that end, knowledge regarding the dependency of the constants in Lemma~\ref{lem;MaxSobRegTestver} on $p$ is indispensable.
Lemma~\ref{lem;MaxSobRegTestver} asserts that the function $\widetilde{K}$ given there is continuous at $2$, so that we can find 
\begin{equation}\label{choices}
  \eps \in \left(0, \frac{\mu(x_0)}{8}\right)
  \text{ and }
  p \in (1, 2)
  \quad \text{such that} \quad
  \frac{\mu(x_0)}{2} - 4\ep
  - \frac{(p-1) \ure^{p \kappa_0 T} \widetilde{K}(p+1)}{p \ep^p} \ge 0,
\end{equation}
where $\kappa_0 \defs \|\kappa\|_{L^\infty(\Omega)}$,
and then let (with $d$ as in \eqref{eq:ass_x0})
\begin{equation}\label{eq:varphip}
  \begin{cases}
    \varphi_p \text{ be a function given by Lemma~\ref{lem;CutOffF}} \\
    \hphantom{\varphi_p} \text{ with } A \defs \ol{\balld} \cap \Ombar,\, V \defs \ballc \cap \ol{\Omega} \text{ and } \eta \defs \frac{1}{2(p+1)}.
  \end{cases}
\end{equation}

\begin{lem}\label{lem;DifIneqforU}
  Assume \eqref{eq:ass_x0}, \eqref{choices} and \eqref{eq:varphip}.
  Then there is $C>0$ such that 
  \begin{align}\label{eq:DifIneqforU:ineq}
    \frac 1p  \ddt  \io \phip u^p &\le 
    \lp{\infty}{\kappa} \io \phip u^p  
    + \frac{p-1}{p\ep^p} \io \phip |\Delta v|^{p+1} 
    -  \io (\mu - 4\ep) \phip u^{p+1} +C
  \end{align}
  in $(0, T)$.
\end{lem}
\begin{proof}
  We proceed as in \cite[Lemma~4.1]{BlackEtAlPossiblePointsBlowup2023}, except that we are unable to replace $-\Delta v$ by $u - v$.
  Testing the first equation in \eqref{P} with $\varphi u^{p-1}$
  and integrating by parts multiple times (the boundary terms vanish as $∂_{ν} φ_p=0=∂_{ν} u = ∂_{ν}v$ 
  on $\partial \Omega$ by \eqref{P} and \eqref{Condi;testfunction}) give
  \begin{align*}
    \frac 1p \ddt \io \phip u^p 
    & = 
    -(p-1) \io \phip u^{p-2}|\nabla u|^2 + \frac 1p \io u^p \Delta \phip 
    - \left(1-\frac 1p \right) \io \phip u^p \Delta v 
    \\ 
    &\pe
    + \frac 1p \io u^p \nabla \phip \cdot \nabla v 
    + \io \kappa \phip u^p - \io \mu \phip u^{p+1} 
  \end{align*}
  in $(0, T)$. 
  Another integration by parts and Young's inequality show
  \begin{align}\label{eq:diff_ineq_u:2}
          \frac 1p \io u^p \nabla \phip \cdot \nabla v 
    &=    - \io u^{p-1} v \nabla u \cdot \nabla \phip 
          - \frac 1p \io u^p v \Delta \phip \notag \\
    &\le  (p-1) \io \phip u^{p-2} |\nabla u|^2
          + \ep \io \phip u^{p+1}
          + c_1 \io \frac{|\nabla \phip|^{2(p+1)}}{\phip^{2p+1}}v^{2(p+1)} \notag \\
    &\pe  + \ep \io \phip u^{p+1}
          + c_1 \io \frac{|\Delta \phip|^{p+1}}{\phip^p}v^{p+1}
  \end{align}
  in $(0, T)$ for some $c_1 > 0$.
  Since $\eta = \frac{1}{2(p+1)}$ by \eqref{eq:varphip},
  \eqref{eq:cutoff:D_varphi_est} implies $|\nabla \phip|^{2(p+1)} \le C_\varphi^{2(p+1)} \phip^{2(p+1)-1}$ and $|\Delta \phip|^{p+1} \le C_\varphi^{p+1}  \phip^{p+1-1}$ in $\Ombar$.
  Together with \eqref{LpForV} this entails boundedness of the third and last integral on the right-hand side in \eqref{eq:diff_ineq_u:2} in $(0, T)$,
  whence there is $c_2 > 0$ such that
  \begin{align}\label{eq:diff_ineq_u:3}
    \frac 1p \ddt \io \phip u^p 
    & \le
    \|\kappa\|_{L^\infty(\Omega)} \io \phip u^p
    + \frac 1p \io u^p \Delta \phip 
    - \left(1-\frac 1p \right) \io \phip u^p \Delta v \notag \\
    &\pe
    - \io (\mu - 2\eps) \phip u^{p+1} 
    + c_2
  \end{align}
  in $(0, T)$. 
  By two additional applications of Young's inequality, we further obtain
  \[
    \frac 1p \io u^p \Delta \phip \le  \ep\io \phip u^{p+1} + c_3
  \] 
  and
  \begin{align*}
    - \left(1-\frac 1p \right) \io \phip u^p \Delta v 
    &\le \ep \io \phip u^{p+1} + 
    \frac{p^p}{\ep^p (p+1)^{p+1}}\left(\frac{p-1}{p}\right)^{p+1}\io \phip |\Delta v|^{p+1} \\
    &\le \ep \io \phip u^{p+1} + 
    \frac{p-1}{p\eps^p} \io \phip |\Delta v|^{p+1}
  \end{align*}
  in $(0,T)$ for some $c_3 > 0$,
  which when combined with \eqref{eq:diff_ineq_u:3} yield \eqref{eq:DifIneqforU:ineq} for some $C > 0$.
\end{proof}
%

The second term on the right-hand side in an integrated version of \eqref{eq:DifIneqforU:ineq} can be absorbed by the last one therein
due to maximal Sobolev regularity (Lemma~\ref{lem;MaxSobRegTestver}) and our choices of $p$ and $\eps$ in \eqref{choices}.
Thus, we obtain the following.
\begin{lem}\label{lem;LpForU}
  Assume \eqref{eq:ass_x0}, \eqref{choices} and \eqref{eq:varphip}.
  Then there exists $C > 0$ such that
  \[
    \io \phip u^p(\cdot, t) \le C
    \qquad \text{for all $t \in (\tau, T)$}.
  \]
\end{lem}
\begin{proof}
  Integrating \eqref{eq:DifIneqforU:ineq} and abbreviating $\kappa_0 \defs \|\kappa\|_{L^\infty(\Omega)}$ gives
  \begin{align}
          \io \phip u^p(\cdot, t) 
    &\le  \ure^{p \kappa_0 (t-\tau)} \int_\Omega \phip u^p (\cdot,\tau)
          - p \int_\tau^t \ure^{p\kapp (t-s)} \io (\mu - 4\ep)\phip u^{p+1}(\cdot, s) \ds \nn\\
    &\pe  + \f{p-1}{ε^p} \int_\tau^t \ure^{p\kapp (t-s)} \io \phip |\Delta v(\cdot, s)|^{p+1} \ds
          + p c_1 \int_\tau^t \ure^{p \kapp (t-s)} \ds \nn\\
    &\le  \ure^{p \kappa_0 T} \int_\Omega \phip u^p (\cdot,\tau)
          - p \int_\tau^t \io (\mu - 4\ep)\phip u^{p+1}(\cdot, s) \ds \nn\\
    &\pe  + \f{(p-1)\ure^{p \kappa_0 T}}{ε^p} \int_\tau^t \io \phip |\Delta v(\cdot, s)|^{p+1} \ds
          + p c_1 T \ure^{p \kappa_0 T}\label{estimate:phiup}
  \end{align}
  for all $t\in (\tau,T)$ and some $c_1 > 0$, where we have used that, by \eqref{choices} and \eqref{eq:varphip}, $(μ-4ε)φ_p\ge 0$.
  Applying Lemma~\ref{lem;MaxSobRegTestver} (with $\widetilde K$ as in that lemma)
  and noting that $\varphi_p^\frac{1}{p+1} \in C^2(\Ombar)$ by Lemma~\ref{lem;CutOffF}, we have 
  \begin{align}
I\defs          \int_\tau^t \io \phip |\Delta v|^{p+1}
    &\le  \widetilde{K} (p+1) \|\phip^{\frac{1}{p+1}} v(\cdot,\tau)\|_{W^{2,p+1}(\Omega)}^{p+1} \nn\\
    &\pe  + \widetilde{K}(p+1)\int_\tau^t \int_\Omega \left(|\nabla \phip^{\frac{1}{p+1}}\cdot \nabla v|^{p+1} + |v \Delta \phip^{\frac{1}{p+1}}|^{p+1} + |\phip^{\frac{1}{p+1}} u|^{p+1}\right) \nn\\
    &\le  \widetilde{K} (p+1) \|\phip^{\frac{1}{p+1}} v(\cdot,\tau)\|_{W^{2,p+1}(\Omega)}^{p+1}
          + \widetilde{K}(p+1) \int_\tau^t \int_\Omega \phip u^{p+1} \nn\\
    &\pe  + \widetilde{K}(p+1)\int_\tau^T \int_\Omega \left(|\nabla \phip^{\frac{1}{p+1}}\cdot \nabla v|^{p+1} + |v \Delta \phip^{\frac{1}{p+1}}|^{p+1} \right). \nn\\
\intertext{Because 
       $v \in L^{p+1}((0, T); W^{1,p+1}(V))$ with $V \defs \ballc \cap \Ombar$ by Lemma~\ref{lem;EstiForV} and Lemma~\ref{lem;L4-epEstiForV}
   and $\operatorname{supp} \varphi \subseteq V$ by \eqref{eq:varphip} and \eqref{Condi;testfunction}, this yields}
    I &\le  \widetilde{K}(p+1) \int_\tau^t \int_\Omega \phip u^{p+1}  + c_2 \label{estimate:intphideltav}
  \end{align}
  for some $c_2>0$.
  Thanks to \eqref{choices}, by inserting \eqref{estimate:intphideltav} into \eqref{estimate:phiup} we can conclude with some $c_3>0$ that 
  \begin{align*}
          \io \phip u^p(\cdot, t)
    &\le  - p \int_\tau^t \ure^{-p\kapp s} \io \left( \mu - 4\ep - \frac{(p-1) \ure^{p \kappa_0 T} \widetilde{K}(p+1)}{p\ep^p} \right) \phip u^{p+1} 
          + c_3
    \le   c_3
  \end{align*}
  for all $t \in (\tau, T)$. 
\end{proof}

\subsection{Localized \tops{$L^\infty$}{L infty} bounds for \tops{$u$}{u}: proof of Theorem~\ref{mainthm1}}
On the basis of Lemma~\ref{lem;LpForU} and without the need for another cutoff argument, we can obtain higher integrability properties of $\nabla v$ than those already asserted by \eqref{LqForNablaV}.
\begin{lem}\label{lem;LqForV}
  Assume \eqref{eq:ass_x0}.
  Then there are $q > 2$ and $C > 0$ such that
  \begin{align}\label{eq:LqForV:statement}
    \| \nabla v (\cdot,t)\|_{L^q(\balld \cap \Omega)} \le C
    \qquad \text{for all $t \in (\tau, T)$}.
  \end{align}
\end{lem}
\begin{proof}
  We let $p$ and $\phip$ be as in \eqref{choices} and \eqref{eq:varphip}.
  Straightforward calculations show that $\phip v$ solves
  \begin{align*}
    \begin{cases}
      (\phip v)_t = \Delta (\phip v)  - \phip v + \phip u - 2\nabla \phip\cdot \nabla v - v \Delta \phip & \text{in $\Omega\times(\tau,T)$}, \\ 
      \partial_\nu (\phip v) = 0 & \text{on $\partial \Omega\times(\tau,T)$},
   \end{cases}
  \end{align*}
  hence Duhamel's formula entails that 
  \[
      \phip v(\cdot,t)
   =  \ure^{(t-\tau) (\Delta-1)}[\phip v(\cdot, \tau)]
      + \int_\tau^t \ure^{(t-s)(\Delta-1)} f(\cdot, s) \ds
  \]
  for all $t\in (\tau,T)$, where $f \defs \phip u - 2\nabla \phip\cdot \nabla v - v \Delta \phip \in L^\infty((\tau, T); L^p(\Omega))$ by Lemma~\ref{lem;EstiForV} and Lemma~\ref{lem;LpForU}.
  Since $p > 1$ by \eqref{choices}, we may fix $q \in (2, \frac{2p}{2-p})$ and then infer from well-known semigroup estimates (cf.\ \cite[Lemma~1.3]{WinklerAggregationVsGlobal2010}) that with some $c_1, c_2 > 0$, 
  \begin{align*}
          \lp{q}{\nabla (\phip v(\cdot,t))} 
    &\le  \|\nabla \ure^{(t-\tau) (\Delta-1)} (\phip v_0)\|_{L^q(\Omega)}
          + \int_\tau^t \|\nabla \ure^{(t-s) (\Delta-1)} f(\cdot, s)\|_{L^q(\Omega)} \ds \\
    &\le  c_1 \|\nabla (\phip v_0)\|_{L^q(\Omega)}
          + c_1 \|f\|_{L^\infty((\tau, T); L^p(\Omega))} \int_0^\infty (1 + s^{-\frac12 - (\frac1p - \frac1q)}) \ure^{-s} \ds
     \le  c_2
  \end{align*}
  for all $t \in (\tau, T)$ since $-\frac12 - (\frac1p - \frac1q) > - 1$.
  Recalling that $\phip = 1$ in $\balld$, we see that 
  \[
    \|\nabla v(\cdot,t)\|_{L^{q}(\balld \cap \Omega)} 
    = \|\nabla (\phip v(\cdot,t))\|_{L^{q}(\balld \cap \Omega)} 
    \le \lp{q}{\nabla (\phip v(\cdot,t))}
    \le c_2
  \]
  for all $t\in (\tau, T)$.
\end{proof}
%

With Lemma~\ref{lem;LqForV} at hand, we may now apply another semigroup argument to conclude boundedness of $u$ near $x_0$ and hence obtain Theorem~\ref{mainthm1}.
\begin{lem}\label{lem;LinftyForU}
  Assume \eqref{eq:ass_x0}. 
  Then there exists $C > 0$ such that
  \[
    \|u(\cdot, t)\|_{L^\infty(\balle \cap \Omega)} \le C
    \qquad \text{for all $t \in (\tau, T)$}
  \]
\end{lem}
\begin{proof}
  As the proof is only based on the first equation in \eqref{P} and the bound \eqref{eq:LqForV:statement},
  this can be shown exactly as in \cite[Lemma~5.2]{BlackEtAlPossiblePointsBlowup2023} where the parabolic--elliptic analogue of \eqref{P} was treated.
  Nonetheless, we choose to give a short proof for the sake of completeness.

  We let $q > 2$ be as given by Lemma~\ref{lem;LqForV}, $\lambda \in (2, q)$
  and $\varphi$ be as given by Lemma~\ref{lem;CutOffF} with $A \defs \ol{\balle} \cap \Ombar$, $V \defs \balld \cap \Ombar$ and $\eta \defs \min \{\frac 1{2\lambda}, \frac{q - \lambda}{q\lambda}\} < \frac12$. 
  Noting that $\varphi u$ satisfies 
  \begin{align*}
    \begin{cases}
      (\varphi u)_t \le \Delta (\varphi u) - u\Delta \varphi -2\nabla \cdot (u\nabla \varphi) - \nabla \cdot (\varphi u \nabla v) + u\nabla \varphi \cdot \nabla v + \kappa \varphi u & \text{in $\Omega\times(\tau,T)$}, \\
      \partial_\nu (\varphi u) = 0 & \text{on $\partial \Omega\times(\tau,T)$}
   \end{cases}
  \end{align*}
  and that $-\frac12 - (\frac1\lambda - \frac1\infty) > - 1$ holds,
  we obtain from the maximum principle and $L^p$-$L^q$ estimates (for which we again refer to \cite[Lemma~1.3]{WinklerAggregationVsGlobal2010}) that with some $c_1>0$, 
  \begin{align*}
  \lp{\infty}{\varphi u(\cdot,t)} & \le 
    \lp{\infty}{u(\cdot, \tau)} 
    + c_1 \sup_{s\in (\tau, t)} \lp{\lambda}{u(\cdot,s)\Delta \varphi + u(\cdot,s) \nabla \varphi \cdot \nabla v(\cdot,s) + \kappa \varphi u(\cdot,s)} \\ 
    &\pe
    + c_1 \sup_{s\in (\tau, t)} \lp{\lambda}{2u(\cdot,s)\nabla \varphi + \varphi u(\cdot,s) \nabla v(\cdot,s)}
  \end{align*}
  for all $t\in (\tau,T)$. 
  For every $s\in (1,∞)$ and $γ\in(0,\f1s]$ there is $c_2(s,γ)>0$ such that 
  \begin{equation}\label{lastproof}
   \|\varphi^{1-\gamma} u\|_{L^s(\Omega)}\le \normj[\Lom s]{u^{γ}(φu)^{1-γ} }\le \normj[\Lom s]{u^{γ}}\normj[\Lom\infty]{φu}^{1-γ} \le c_2(s, \gamma) \|\varphi u\|_{L^\infty(\Omega)}^{1-\gamma}
  \end{equation}
  for every $t\in(0,T)$, because of \eqref{L1ForU}. By Hölder's inequality, with 
  $C_\varphi$ taken from Lemma~\ref{lem;CutOffF},
  and because of \eqref{lastproof} and Lemma~\ref{lem;LqForV}, for some $c_3>0$ we have
  \begin{align*}
    &\pe \lp{\lambda}{u\Delta \varphi + u \nabla \varphi \cdot \nabla v + \kappa \varphi u} \\ 
    &\le  C_\varphi \lp{\lambda}{\varphi^{1-2\eta} u}
          + C_\varphi \lp{\frac{q\lambda}{q-\lambda}}{\varphi^{1-\eta} u} \|\nabla v \|_{L^q(V)}
          + \lp{\infty}{\kappa}^{1-\frac 1\lambda}\lp{\infty}{\varphi u}^{1-\frac 1\lambda} \lp{1}{\kappa \varphi u}^{\frac 1\lambda} \\ 
    &\le  c_3 \left( \lp{\infty}{\varphi u}^{1-2\eta} + \lp{\infty}{\varphi u}^{1-\eta} + \lp{\infty}{\varphi u}^{1-\frac 1\lambda} \right)
  \end{align*}
  and 
  \begin{align*}
          \lp{\lambda}{2u \nabla \varphi + \varphi u \nabla v } 
    &\le  2C_\varphi \lp{\lambda}{\varphi^{1-\eta} u}
          + \lp{\infty}{\varphi u}^{1-\frac{q-\lambda}{q\lambda}} \lp{1}{\varphi u}^{\frac{q-\lambda}{q\lambda}} \|\nabla v \|_{L^q(V)} \\
    &\le  c_3 \left( \lp{\infty}{\varphi u}^{1-\eta} + \lp{\infty}{\varphi u}^{1-\frac{q-\lambda}{q\lambda}} \right)
  \end{align*}
  in $(\tau, T)$. 
  Therefore, we can find $\theta \in (0,1)$ and $c_4>0$ such that 
  \[
    \sup_{t \in (\tau, T')} \lp{\infty}{\varphi u(\cdot,t)} \le c_4 \left( 1+ \sup_{t\in (\tau, T')}\lp{\infty}{\varphi u(\cdot,t)}^{\theta} \right)
  \]
  for all $T' \in (\tau, T)$, whenceupon the statement follows by Young's inequality and due to $\varphi = 1$ in $A$.
\end{proof}
%
\begin{proof}[Proof of Theorem~\ref{mainthm1}]
  Lemma~\ref{lem;LinftyForU} and the regularity $u \in C^0(\Ombar \times [0, \tau])$ show (i) with $U \defs \balle$, which directly implies (ii).
\end{proof}

\section*{Acknowledgements}
We would like to thank H.\ Meinlschmidt for pointing out the possibility to adapt our earlier proof of continuity of the optimal constant $p\mapsto K(p, p)$ in the maximal Sobolev regularity statement (cf.\ Lemma~\ref{lem;MaxSobReg}) to continuity of $(p, q) \mapsto K(p, q)$.



\def\cprime{$'$}

\end{document}